\documentclass[12pt]{amsart}
\usepackage{amsmath, amsthm, amscd, amsfonts, amssymb, latexsym, graphicx, color}
\usepackage[bookmarksnumbered, colorlinks, plainpages, hypertex]{hyperref}

\usepackage[cp1250]{inputenc}

\usepackage{amsmath,amsthm}
\usepackage{amssymb,latexsym}
\usepackage{enumerate}

\newtheorem{theorem}{Theorem}[section]
\newtheorem{lemma}[theorem]{Lemma}
\newtheorem{proposition}[theorem]{Proposition}
\newtheorem{corollary}[theorem]{Corollary}
\theoremstyle{definition}
\newtheorem{definition}[theorem]{Definition}
\newtheorem{example}[theorem]{Example}

\newtheorem{construction}[theorem]{Construction}
\theoremstyle{remark}
\newtheorem{remark}[theorem]{Remark}
\numberwithin{equation}{section}

\begin{document}

\title[Definable transformation to normal crossings]
      {Definable transformation to \\
      normal crossings over Henselian fields \\
      with separated analytic structure}

\author[Krzysztof Jan Nowak]{Krzysztof Jan Nowak}

%\subjclass[2000]{Primary 32P05, 32S45; Secondary 14E15, 14G22.}


\date{}

\begin{abstract}
We are concerned with rigid analytic geometry in the general setting of Henselian fields $K$ with separated analytic structure, whose theory was developed by Cluckers--Lipshitz--Robinson. It unifies earlier work and approaches of numerous mathematicians. Separated analytic structures admit reasonable relative quantifier elimination in a suitable analytic language. However, the rings of global analytic functions with two kinds of variables seem not to have good algebraic properties such as Noetherianity or excellence. Therefore the usual global resolution of singularities from rigid analytic geometry is no longer at our disposal. Our main purpose is to give a definable version of the canonical desingularization algorithm (the hypersurface case) due to Bierstone--Milman so that both these powerful tools are available in the realm of non-Archimedean analytic geometry at the same time. It will be carried out within a category of definable, strong analytic manifolds and maps, which is more flexible than that of affinoid varieties and maps. Strong analytic objects are those definable ones that remain analytic over all fields elementarily equivalent to $K$. This condition may be regarded as a kind of symmetry imposed on ordinary analytic objects. The strong analytic category makes it possible to apply a model-theoretic compactness argument in the absence of the ordinary topological compactness. On the other hand, our closedness theorem enables application of resolution of singularities to topological problems involving the topology induced by valuation. Eventually, these three results will be applied to such issues as the existence of definable retractions or extending continuous definable functions. The established results remain valid for strictly convergent analytic structures, whose classical examples are complete, rank one valued fields with the Tate algebras of strictly convergent power series. The earlier techniques and approaches to the purely topological versions of those issues cannot be carried over to the definable settings because, among others, non-Archimedean geometry over non-locally compact fields suffers from lack of definable Skolem functions.
\end{abstract}

\maketitle

\section{Introduction}

We are concerned with rigid analytic geometry in the general setting of Henselian fields $K$ with separated analytic structure (with two kinds of variables: ones vary over the closed unit ball $K^{\circ}$ and the other ones over the open unit ball $K^{\circ \circ}$), whose theory was developed by Cluckers--Lipshitz--Robinson \cite{C-Lip-R,C-Lip-0,C-Lip}. It unifies earlier work and approaches of numerous mathematicians (see e.g.\ \cite{De-Dries,Dries,Lip,Dries-Mac,Dries-Has,L-R-0,L-R}).

\vspace{1ex}

Separated analytic structures, unlike strictly convergent ones, admit reasonable quantifier elimination, relative to the auxiliary sorts, in suitable analytic extensions of the 3-sorted language of Denef--Pas~\cite{Pa1} or 2-sorted language of Basarab--Kuhlmann~\cite{Bas,Ku-1}. However, the rings of global analytic functions with two kinds of variables seem not to have good algebraic properties such as Noetherianity or excellence (cf.~\cite[Remark~5.2.9]{C-Lip-0}). Thus the usual global resolution of singularities from rigid analytic geometry is no longer at our disposal.

\vspace{1ex}

The main purpose of this paper is to give a definable version of the canonical desingularization algorithm due to Bierstone--Milman~\cite{BM} in the hypersurface case (where the concepts of strict and weak transforms coincide). Thus the two powerful tools, quantifier elimination and resolution of singularities, will be available in the realm of non-Archimedean analytic geometry at the same time. The algorithm provides a local invariant such that blowing up its maximum strata leads to desingularization or transformation to normal crossings (op.cit., Theorems~1.6 and~1.10). This will be accomplished in Section~3 within a certain category of definable, strong analytic manifolds and maps, which is introduced and examined in Section~2. One of the essential ingredients of our approach is the closedness theorem from our papers~\cite{Now-1,Now-2,Now-3}. The strong analytic category takes into account all those fields which are elementarily equivalent to the ground field $K$, and makes it possible to apply a model-theoretic compactness argument in the absence of the ordinary topological compactness. It is more flexible than that of affinoid varieties and maps, and allows us to introduce in a geometric way the concepts of a blowup along a smooth strong analytic center and of (weak) transform.

\vspace{1ex}

On the other hand, the closedness theorem enables application of resolution of singularities to topological problems which involve the topology induced by valuation. Eventually, both these results, along with elimination of valued field quantifiers, will be applied in Section~4 to such issues as the existence of definable retractions or extending continuous definable functions, including the theorems of Tietze--Urysohn and Dugundji (cf.~\cite{Now-4,Now-5}).

\vspace{1ex}

Making use of the closedness theorem, local transformation to normal crossings, elimination of valued field quantifiers and relative (to some auxiliary, imaginary, linearly ordered sorts) quantifier elimination for ordered abelian groups, we established in the papers~\cite{Now-1,Now-2,Now-3} a lot of new results as, for instance: piecewise continuity of definable functions, several versions of the \L{}ojasiewicz inequalities and of curve selection over arbitrary, Henselian, non-trivially valued, equicharacteristic zero fields (including the non-algebraically closed ones), as well as many their further applications. Let us emphasize that the \L{}ojasiewicz inequalities and curve selection were known before only in the case of algebraically closed valued fields.

\vspace{1ex}

Observe that the established results remain valid for strictly convergent analytic structures, because every such a structure can be extended in a definitional way to a separated analytic structure (cf.~\cite{C-Lip}). Classical examples of them are complete, rank one valued fields with the Tate algebras of strictly convergent power series.

\vspace{1ex}

Note that our treatment of the problems from Section~4 via strong analytic maps allows us not to appeal to the theory of quasi-affinoid subdomains developed by Lipshitz--Robinson~\cite{L-R-0}. And that the earlier techniques and approaches to the purely topological versions of those problems cannot be carried over to the definable settings because, among others, non-Archimedean geometry over non-locally compact fields suffers from lack of definable Skolem functions. For a more detailed discussion about their classical, purely topological counterparts (see e.g.\ \cite{El-1,El-2,K-K-S}), we refer the reader to our paper~\cite{Now-4}.

\vspace{1ex}

In Section~5, we discuss certain intricacies of non-Archimedean analytic geometry and give some background behind quantifier elimination. Also emphasized are some advantages of the approach proposed in this paper.

\vspace{1ex}

Now, following the papers~\cite{C-Lip-0,C-Lip}, we remind the reader of the concept of an analytic structure. Fix a Henselian, non-trivially valued field $K$ of equicharacteristic zero; $K$ may not be algebraically closed.
Denote by $v$, $\Gamma = \Gamma_{K}$, $K^{\circ}$, $K^{\circ \circ}$ and $\widetilde{K}$ the valuation, its value group, the valuation ring (closed unit ball), maximal ideal (open unit ball) and residue field, respectively. The multiplicative norm corresponding to $v$ will be denoted by $| \cdot |$. The $K$-topology on $K^{n}$ is the one induced by valuation $v$. Observe that the $K$-topology is totally disconnected, and that a closed unit ball is a disjoint union of infinitely many open unit balls (since the field $K$ is not locally compact).

\vspace{1ex}

Let $K$ be a field with a separated analytic $\mathcal{A}$-structure over a separated Weierstrass system $\mathcal{A} = \{ A_{m,n} \}_{m,n \in \mathbb{N}}$, i.e.\ with a collection $\{ \sigma_{m,n} \}_{m,n \in \mathbb{N}}$ of homomorphisms from $A_{m,n}$ to the ring of $K^{\circ}$-valued functions on $(K^{\circ})^{m} \times (K^{\circ \circ})^{n}$. We consider the ground field $K$ in the analytic language $\mathcal{L} = \mathcal{L}_{\mathcal{A}}$ from~\cite{C-Lip-0}. It is a two sorted, semialgebraic language $\mathcal{L}_{Hen}$, augmented on the main, valued field sort $K$ by the multiplicative inverse $(\cdot)^{-1}$ and the names of all functions of the collection $\mathcal{A}$, together with the induced language on the auxiliary sort $RV(K)$. Power series $f$ from $A_{m,n}$ are construed as $f^{\sigma}=\sigma(f)$ via the analytic $\mathcal{A}$-structure on their natural domains and as zero outside them.

\vspace{1ex}

Without changing the family of definable sets, one can assume that the homomorphism $\sigma_{0,0}$ from $A_{0,0}$ into $K^{\circ}$ is injective (whence so are the homomorphisms $\sigma_{m,n}$), which will be adopted in the sequel. Recall further that the field $K$ has $\mathcal{A}(K)$ structure by extension of parameters (cf.~\cite[Section~4.5]{C-Lip-0} and also \cite[Section~2]{Now-3}); more generally, $K$ has $\mathcal{A}(F)$ structure for any subfield $F$ of $K$.
Let $T_{\mathcal{A}}$ be the $\mathcal{L}_{\mathcal{A}}$-theory of all Henselian, non-trivially valued fields of equicharacteristic zero with analytic $\mathcal{A}$-structure.

\vspace{1ex}

The field $K$ is embeddable into every model $L$ of the $\mathcal{L}_{\mathcal{A}(K)}$-theory of $K$. Under the above assumption of injectivity, the field $L$ has $\mathcal{A}(F)$-structure for any subfield $F$ of $L$.
By abuse of notation, we shall then identify the power series $f \in A_{m,n}^{\dag}(K) := K \otimes_{K^{\circ}} A_{m,n}(K)$ with their interpretations $f^{\sigma}$ on their natural domains. The rings $A_{m,n}^{\dagger}(K)$ of global analytic functions seem to suffer from lack of good algebraic properties. Only the rings $A_{m,0}^{\dagger}(K)$ and $A_{0,n}^{\dagger}(K)$ of power series with one kind of variables enjoy very good algebraic properties being, for instance, Noetherian, factorial, normal and excellent (as they fall under the Weierstrass--R\"{u}ckert theory; cf.~\cite[Section~5.2]{C-Lip-0} and~\cite[Section~5.2]{B-G-R}). Therefore, the techniques of resolution of singularities by Bierstone--Milman~\cite{BM} or Temkin~\cite{Tem-2} cannot be directly applied to them (and thus on the global space $M_{0} = (K^{\circ})^{m} \times (K^{\circ \circ})^{n}$).

\vspace{1ex}

We shall show that the output data of the algorithm are strong analytic if so are the input data and, consequently, that the desingularization invariant takes only finitely many values and its equimultiple loci are definable. Actually, the resolution process works for arbitrary strong analytic functions on strong analytic manifolds.
Our approach, pursued in Section~3, is based on analysis of the data from which the invariant is built, which in turn relies principally on the following four crucial points:

\vspace{1ex}

1. The functions and submanifolds involved in the resolution process, are definable and strong analytic. Consequently, via a model-theoretic compactness argument, the orders of those functions are definable, i.e.\ their equimultiple loci are finite in number and definable. This enables further analysis of the entries $\nu_{r}(a)$ of the invariant, which are a kind of higher order, rational multiplicities of certain strong analytic functions. Hence and by the canonical character of the process, the successive centers of blowups, being the maximum strata of the desingularization invariant, are definable and strong analytic.

\vspace{1ex}

2. The entries $\nu_{r}(a)$ can be defined by computations which involve orders of vanishing in suitable local coordinates (independently of their choice) induced by generic affine coordinates of the ambient affine space. Therefore, such computations can be performed through suitable definable families of coordinates induced by affine coordinates. This is of great importance, especially in the absence of definable Skolem functions. Hence $\nu_{r}(a)$ turn out to be definable, i.e.\ their equimultiple loci are finite in number and definable.

\vspace{1ex}

3. Making use of the closedness theorem, it is possible to partition each ambient manifold, achieved by blowing up, into a finite number of definable clopen pieces so that, on each of them, both the exceptional hypersufaces (which reflect the history of the process and enable the further construction of the desingularization invariant) and next the successive blowup, can be described in a definable geometric way. This geometric bypass compensate for inability to globally describe the centers of the successive blowups in a purely analytic way, which is caused by lack of good algebraic properties of the rings of global analytic functions.

\vspace{1ex}

4. The canonical algorithm depends only on the completions of the local rings of analytic function germs at the points of the ambient manifolds. Therefore, finite partitions of those manifold into definable clopen pieces do not affect its output data, although quasi-affinoid structure may change. This legitimizes partitions indicated above.

\section{Strong analyticity: blowups and (weak) transforms}

Strong analyticity is a model-theoretic strengthening of the weak concept of analyticity determined by a given separated Weierstrass system (treated in the classical case e.g.~by Serre~\cite{Se}), which works well within the definable settings. By strong analytic functions and manifolds, we mean the analytic ones that are definable in the structure $K$ and remain analytic in each field $L$ elementarily equivalent to $K$ in the language $\mathcal{L}_{\mathcal{A}(K)}$. Examples of such functions and manifolds are those obtained by means of the implicit function theorem and the zero loci of strong analytic submersions.

\vspace{1ex}

From now on, "definable" will mean "$\mathcal{L}_{\mathcal{A}(K)}$-definable". Since all analytic functions and manifolds occurring in the resolution process turn out to be strong analytic, the words "definable, strong analytic" will be shortened for simplicity to "analytic".

\vspace{1ex}

Let $f: M \to K$ be an analytic function on an analytic manifold $M$ (by the above convention, in the category of strong analytic manifolds). By $\mathrm{supp}\, f$, the support of $f$, we mean the closure (in the $K$-topology) of the complement of its zero locus $V(f)$. It is not difficult to check that $\mathrm{supp}\, f$ is a clopen definable subset and that the order of vanishing $\mu_{a}(f) \in \mathbb{N}$ of $f$ at a point $a \in \mathrm{supp}\, f$ is finite; obviously, $\mu_{a}(f) = \infty$ iff $a \in \mathrm{supp}\, f$. The following basic result on order of vanishing will often be used in the sequel.

\begin{proposition}\label{order}
Under the above assumptions, the set of orders of vanishing $\{ \mu_{x}(f): a \in M \}$ is finite. Moreover, the conclusion remains true for definable families of analytic functions
\end{proposition}

\begin{proof}
The assertion follows directly, via a routine model-theoretic compactness argument, from the assumption of strong analyticity.
\end{proof}

\begin{remark}
It follows immediately from the proposition that the sets of orders of vanishing $\{ \mu_{x}(f): a \in M \}$ over the fields $L$ elementarily equivalent
to $K$ in the language $\mathcal{L}_{\mathcal{A}(K)}$ coincide.
\end{remark}

Let $C$ be a closed analytic submanifold of $M$. Likewise in the classical case, we can define the order of the analytic function $f$ along $C$ at a point $a \in C$ by putting
$$\mu_{C,a}(f) := \min \, \{ \mu_{x}(f): \, x \in C \ \text{near} \ a \}. $$
Then $\mu_{C,a}(f)$ takes only finitely many values and is constant on clopen definable subsets $F_{k}$. Using the closedness theorem, it is not difficult to show the following

\begin{proposition}\label{ord-part}
Under the above assumptions, there are a finite number of pairwise disjoint, clopen definable subsets $\Omega_{k}$ of $M$ covering $C$ such that the order of vanishing $\mu_{C,a}(f)$ is constant on $C \cap \Omega_{k}$. Further, we can ensure that $f$ vanishes on $\Omega_{k}$ if $\mu_{C,a}(f) = \infty$ on $C \cap \Omega_{k}$.  \hspace*{\fill} $\Box$
 \end{proposition}

\vspace{1ex}

{\em A definable construction of the blowup along $C$.} Suppose $C$ is a closed (definable) analytic submanifold of $M_{0} = (K^{\circ})^{k} \times (K^{\circ \circ})^{l}$ of dimension $p \leq n-1$ with $n = k+l$, and consider the finite subsets $I$ of $\{ 1,\ldots,n \}$ of cardinality $p$. Define the canonical projections $\phi_{I}$ onto the $x_{I}$ variables by putting
$$ (\phi_{I}(x))_{i} = \left\{ \begin{array}{cl}
                        x_{i} & \mbox{ if } i \in I, \\
                        0 & \mbox{ otherwise }.
                        \end{array}
               \right.
$$
Let $U_{I}$ be the set of all points $a \in C$ at which the restriction of $\phi_{I}$ to $C$ is an immersion. Obviously, the sets $U_{I}$ are a finite open definable covering of $C$. By \cite[Corollary~2.4]{Now-4} (being a consequence of the closedness theorem), there exists a finite clopen definable partition $\Omega_{I}$ of $C$ such that $\Omega_{I} \subset U_{I}$. For a fixed $I$, denote  the restriction of $\phi_{I}$ to $\Omega_{I}$ by $\phi$. The fibres of $\phi$ are of course finite. It follows from the closedness theorem that $\phi$ is a definably closed map. Therefore, for any $a \in \Omega_{I}$  and a neighbourhood $U$ of $\phi^{-1}(\phi(a))$, there exists a neighbourhood $V$ of $\phi(a)$ such that $\phi^{-1}(V) \subset U$. Hence and by the implicit function theorem, the set
$$ F := \{ (a,z) \in \Omega_{I} \times M_{0}: \ \exists \: x \in \Omega_{I} \ x \neq a, \ \phi(x) = \phi(a), \ z=x-a \} $$
is a closed definable subset of $\Omega_{I} \times M_{0}$. It follows from the closedness theorem, that there is an $\epsilon \in |K |$, $\epsilon > 0$, such that $|a-x|> \epsilon$ for every $a,x \in \Omega_{I}$ with $\phi(a) = \phi(x)$ and $a \neq x$. Then the restriction of $\phi_{I}$ to the clopen subset
$$ B(a,\epsilon) \cap \phi^{-1}(\phi(\Omega_{I} \cap B(a,\epsilon))) $$
is injective whence a bianalytic map onto the clopen image (by the closedness theorem again). Put
$$ \Omega_{I}^{*} := \bigcup_{a \in \Omega_{I}} \ (B(a,\epsilon) \cap \phi_{I}^{-1} (\phi_{I}(\Omega_{I} \cap B(a,\epsilon)))). $$

\vspace{1ex}

\begin{lemma}
$\Omega_{I}^{*}$ is a closed subset of $M_{0}$ whence a clopen neighbourhood of $\Omega_{I}$.
\end{lemma}

\begin{proof}
Indeed, take a point $b$ from the closure of $\Omega_{I}^{*}$. Then
$$ B(b,\delta) \cap \Omega_{I}^{*} \neq \emptyset $$
for every $\delta < \epsilon$. Further,
$$ B(b,\delta) \cap B(a,\delta) \cap \phi_{I}^{-1} (\phi_{I}(\Omega_{I} \cap B(a,\epsilon))) \neq \emptyset $$
for some $a \in \Omega_{I}$. Hence
$$ B(b,\delta) \cap \phi_{I}^{-1} (\phi_{I}(\Omega_{I} \cap B(b,\epsilon))) \neq \emptyset $$
and
$$ \phi_{I}(B(b,\delta) \cap \phi_{I}(\Omega_{I} \cap B(a,\epsilon)) \neq \emptyset, \ \ B(\phi_{I}(b),\delta) \cap \phi_{I}(\Omega_{I} \cap B(a,\epsilon)) \neq \emptyset. $$
Thus $\phi_{I}(b)$ lies in the closure of $\phi_{I}(\Omega_{I} \cap B(a,\epsilon))$. Since, by the closedness theorem, this is a closed subset, we get $\phi_{I}(b) \in \phi_{I}(\Omega_{I} \cap B(a,\epsilon))$. Therefore $\phi_{I}(b) = \phi_{I}(c)$ for some $c \in \Omega_{I} \cap B(b,\epsilon)$. Then
$$ b \in B(c,\epsilon) \cap \phi_{I}^{-1} (\phi_{I}(\Omega_{I} \cap B(c,\epsilon)))), $$
and thus $b \in \Omega_{I}^{*}$, as required.
\end{proof}

\vspace{1ex}

We can regard the restrictions of $\phi_{I}$ to $\Omega_{I}^{*}$ as coordinate charts on $C$ in analogy to regular coordinate charts considered in~\cite[Section~3]{BM}. What will play an important role is that every $\Omega_{I}$ is the zero locus of an analytic submersion $\theta: \Omega_{I}^{*} \to K^{q}$ defined as follows. Let $J := \{1,\ldots,n \} \setminus I$ and $\psi_{J}$ be the canonical projection onto the variables $x_{J}$. Any point $x \in \Omega_{I}^{*}$ lies in
$$ B(a,\epsilon) \cap \phi_{I}^{-1} (\phi_{I}(\Omega_{I} \cap B(a,\epsilon))) $$
for some $a \in \Omega_{I}$. Take a unique $y \in \Omega_{I} \cap B(a,\epsilon)$ such that $\phi_{I}(x) = \phi_{I}(y)$ and set $\theta(x) := \psi_{J}(x) - \psi_{J}(y)$.
This construction enables the standard definition of the blowup of $\Omega_{I}^{*}$ along $C$ which is an analytic submanifold of $\Omega_{I}^{*} \times \mathbb{P}^{q-1}(K)$.

\vspace{1ex}

We have thus constructed the blowups along $C$ of the pairwise disjoint, clopen definable neighbourhoods $\Omega_{I}^{*}$ of the clopen pieces $\Omega_{I}$ of the submanifold $C$ and the blowups of those neighbourhoods. On the complement
$$ M_{0} \setminus \bigcup_{I} \; \Omega_{I}^{*}, $$
which is a clopen subset of $M_{0}$, it suffices to define the blowup to be the identity. By gluing, we obtain the blowup $\sigma_{1}: M_{1} \to M_{0}$, which is an analytic map, where $M_{1}$ is an analytic submanifold of $M_{0} \times \mathbb{P}^{q-1}(K)$ and $q = n-p$ is the codimension of $C$ in $M_{0}$. We should once again emphasize that this construction is performed in the category of strong analytic manifolds and maps.

\vspace{1ex}

\begin{remark}\label{proj}
In this manner, the blowup $M_{1}$ can be further analyzed by using the $q$ standard affine clopen charts on $\mathbb{P}^{q-1}(K)$.
\end{remark}

The above construction leads to the following

\begin{definition}
Let $M$ be a closed analytic submanifold of dimension $m$ of $M_{0} : =(K^{\circ})^{k} \times (K^{\circ \circ})^{l}$, $\phi_{1},\ldots,\phi_{n}$ be affine coordinates on $M_{0}$ with $n :=k+l$, $\Omega^{*}$ be a clopen definable subset of $M_{0}$ and $\Omega := \Omega^{*} \cap M$. We say that $\phi_{1},\ldots,\phi_{m}$ are a definable coordinate system for $M$ on $\Omega^{*}$ if the restriction of $(\phi_{1},\ldots,\phi_{m})$ is an immersion of $\Omega$ such that for each point $x \in \Omega^{*}$ there is a unique point $y \in \Omega$  that is closest to $x$ from among $(\phi_{1},\ldots,\phi_{m})^{-1}(x) \cap \Omega$. We then call $\Omega$ a definable chart with coordinates $\phi_{1},\ldots,\phi_{m}$ on $\Omega^{*}$. As demonstrated above, $\Omega$ is then the zero locus of an analytic submersion $\theta: \Omega^{*} \to K^{n-m}$.
\end{definition}

Summing up, we have proven the following

\begin{proposition}
Every closed analytic submanifold $M$ of dimension $m$ in $M_{0} = (K^{\circ})^{k} \times (K^{\circ \circ})^{l}$ can be partitioned into a finite number of pairwise disjoint, clopen definable charts $\Omega$ with coordinates on clopen subsets $\Omega^{*}$ of $M_{0}$.  Moreover, $M \cap \Omega^{*}$ is the zero locus of an analytic submersion $\theta: \Omega^{*} \to K^{n-m}$.   \hspace*{\fill} $\Box$
\end{proposition}

\begin{corollary}\label{zero-loc}
Let $C \subset M$ be two closed analytic submanifolds of $M_{0}$ of dimension $p$ and $m$, respectively. Then there exist a finite number of pairwise disjoint, clopen definable subsets $U_{s}$ of $M$ which cover $C$ and such that $C \cap U_{s}$ are the zero loci of some analytic submersions $\theta_{s}: U_{s} \to K^{m-p}$. In particular, if $C$ is a hypersurface in $M$, then $C \cap U_{s} = V(\theta_{s})$ for some analytic submersions $\theta_{s}: U_{s} \to K$.   \hspace*{\fill} $\Box$
\end{corollary}

Using the above methods, we can obtain, in the category of strong analytic manifolds and maps, the following characterization of normal crossing divisors, the detailed verification being left to the reader.

\begin{corollary}
Let $f: M \to K$ be a analytic function on an analytic submanifold of $M_{0}$. If $f$ is a normal crossing divisor (in the usual sense), then there exists a finite partition of $M$ into clopen definable subsets $\Omega_{s}$ and, for each $s$, analytic submersions
$$ \theta_{s_{1}}, \ldots, \theta_{s,l_{s}} : \Omega_{s} \to K \ \ \text{and} \ \ k_{1},\ldots,k_{l_{s}} \in \mathbb{N} $$
such that
$$ f \sim \theta_{s_{1}}^{k_{1}} \cdot \ldots \cdot \theta_{s,l_{s}}^{k_{l_{s}}} \, ; $$
here $\sim$ means equal up to an analytic unit.   \hspace*{\fill} $\Box$
\end{corollary}

In view of the foregoing, we can readily construct in a definable way the transform of an analytic hypersurface as well.

\begin{construction}
Consider a blowup $\sigma_{1}: M_{1} \to M_{0}$ along smooth analytic center $C$ and with exceptional hypersurface $E$. Let $X$ be an analytic hypersurface of $M_{0}$ corresponding to an analytic function $f: M_{0} \to K$; put $f_{1} := f \circ \sigma_{1}$. By Corollary~\ref{zero-loc} and Proposition~\ref{ord-part}, there exist a finite number of pairwise disjoint, clopen subsets $U_{s}$ of $M_{1}$ which cover $E$ and such that $E \cap U_{s} = V(\theta_{s})$ for an analytic submersion $\theta_{s}$ and that the order of vanishing $\mu_{E,a}(f_{1}) = d_{s}$ is constant on $E \cap U_{s}$. Then the transform $X_{1}$ of $X$ is determined on $U_{s}$ by the analytic function
$$ f_{1} := \theta_{s}^{-d_{s}} \cdot f \circ \sigma_{1}; $$
actually, $d_{s}$ is the largest power of $\theta_{s}$ that factors from $g \circ \sigma_{1}$.
\end{construction}

\section{Definable desingularization algorithm}

In this section, the desingularization algorithm by Bierstone--Milman \cite[Chapter~II]{BM} will be adapted to the definable settings. To be brief, for majority of details the reader is referred to their paper. We give a concise outline of the process of transforming an analytic function $g \in A_{k,l}^{\dagger}(K)$ to normal crossings or, equivalently, resolving singularities of the hypersurface $X = X_{0} = V(g)$ of the manifold $M_{0} = (K^{\circ})^{k} \times (K^{\circ \circ})^{l}$ determined by $g$. The notation and terminology related to the local invariant for desingularization will generally follow those from op.cit.

\begin{remark}
The desingularization algorithm, and thus Theorems~\ref{main} and~\ref{main-2} as well, will of course hold whenever $g$ is a definable, strong analytic function on an arbitrary, definable, strong analytic manifold $M_{0}$.
\end{remark}

Consider a sequence of admissible blowups $\sigma_{j}: M_{j} \to M_{j-1}$ along admissible smooth centers $C_{j-1}$, $j=1,2,\ldots$;
let $E_{j}$ denote the set of exceptional hypersurfaces in $M_{j}$ (op.cit., p.~212). Let $X_{1}, X_{2}, X_{3}, \ldots$ denote the successive transforms of the given  hypersurface $X$; here the strict and weak transforms coincide. {\em Admissible} means that $C_{j}$ and $E_{j}$ simultaneously have only normal crossings
and that $\mathrm{inv}_{X}(\cdot)$ is locally constant on $C_{j}$ for all $j$. We can now state the main result, being a definable version of op.cit., Theorem~1.6.

\begin{theorem}\label{main}
Under the above assumptions, there exists a finite sequence of blowups with smooth admissible centers $C_{j}$ such that:

1) for each $j$, either $C_{j} \subset \mathrm{Sing}\, X_{j}$ or $X_{j}$ is smooth and $C_{j} \subset X_{j} \cap E_{j}$;

2) the final transform $X'$ of $X$ is smooth (unless empty), and $X'$ and the final exceptional hypersurface $E'$ simultaneously have only normal crossings.
\end{theorem}

First we begin with the necessary notation:

\vspace{1ex}

$E(a) := \{ H \in E_{j}: a \in H \}$.

\vspace{1ex}

For a point $a = a_{j} \in M_{j}$, let $a_{j-1} \in M_{j-1}, \ldots, a_{0} \in M_{0}$
be the images of $a$ under the successive blowups.

\vspace{1ex}

The order of vanishing of an analytic function germ $f$ at $a$ is $\mu_{a}(f)$.

\vspace{1ex}

In each year $j$, the local invariant $\mathrm{inv}_{X}(a)$ at a point $a \in M_{j}$ is the word:
$$ \mathrm{inv}_{X}(a) = (\nu_{1(a)},s_{1}(a); \nu_{2}(a),s_{2}(a); \ldots; \nu_{t}(a),s_{t}(a);\nu_{t+1}(a)), $$
where $0 < \nu_{1}(a), \ldots ,\nu_{t}(a) \in \mathbb{Q}$, $s_{1}(a), \ldots ,s_{t}(a) \in \mathbb{N}$ and
$\nu_{t+1}(a) = 0$ or $\infty$; note that $t \leq n$ (op.cit., p.~213); $\nu_{1}(a) = \mu_{a}(g)$ where $g$ is a local equation at $a$ of $X$. We consider such words with the lexicographic ordering. The inductive resolution process terminates unless
$0 < \nu_{r}(a) < \infty$.

\vspace{1ex}

The invariant $\mathrm{inv}_{X}(\cdot)$ is upper semicontinuous (i.e.\ each point $a \in X_{j}$ admits
an open neighbourhood U such that $\mathrm{inv}_{X}(x) \leq \mathrm{inv}_{X}(a)$ for all $x \in U$) and infinitesimally upper-semicontinuous (i.e.\ $\mathrm{inv}_{X}(a) \leq \mathrm{inv}_{X}(\sigma_{j}(a)$ for all $j \geq 1$); op.cit., Theorem~1.14.

\vspace{1ex}

An infinitesimal presentation (of codimension $p$) is the following data (op.cit., p.~222):
$$ (N(a), \mathcal{H}(a),\mathcal{E}(a)) $$
where:

$N_{p}(a)$ is a germ at $a$ of a regular submanifold of codimension p;

$\mathcal{H}(a) = \{ (h,\mu_{h}) \}$ is a finite collection of pairs with
$h \in \mathcal{O}_{N,a}$, $\mu_{h} \in \mathbb{Q}$, $0 \leq \mu_{h} \leq \mu_{a}(h)$;

$\mathcal{E}(a)$ is a collection of smooth hypersurfaces $H \ni a$ such that $N$ and $\mathcal{E}(a)$
simultaneously have only normal crossings, and $N \not \subset H$ for all $H \in \mathcal{E}(a)$.

\vspace{1ex}

The equimultiple locus of the infinitesimal presentation is
$$ S_{\mathcal{H}(a)} := \{ x \in N: \mu_{x}(h) \geq \mu_{h} \ \ \forall \, (h,\mu_{h}) \in \mathcal{H}(a) \}; $$
put
$$ \mu_{\mathcal{H}(a)} := \min \left\{ \frac{\mu_{a}(h)}{\mu_{h}} \right\}. $$

\begin{remark}
In view of the canonical character of the resolution process, the maximum loci of the desingularization invariant (being at the same time the centers of the successive blowups) are strong analytic, because locally they are constructed within rigid analytic geometry based on rings with good algebraic properties.
\end{remark}

At this stage we can readily pass to the resolution process. The easiest is the initial year zero before any blowup.

\vspace{1ex}

{\em Year zero.} For each $a \in M_{0}$, we start with the following codimension 0 presentation for the equation $g$:
$$ (N_{0}(a), \mathcal{G}_{1}(a),\mathcal{E}_{1}(a)), \  N_{0}(a) = M_{0}, \ \mathcal{G}_{1}(a) = \{ (g, \mu_{a}(g)) \},
   \ \mathcal{E}_{1}(a) = \emptyset . $$
Put $d := \nu_{1}(a) = \mu_{1}(a)(g)$, $s_{1} := 0$ and $\mathcal{F}_{1}(a) = \mathcal{G}_{1}(a)$. The further definable constructions should take into account
the equimultiple strata of the entry $\nu_{1}$ (and in the further years, the equimultiple strata of the successive entries already constructed). Apply Construction~4.18, op.cit., to get a codimension 1 presentation $\mathcal{H}_{1}(a)$ as explained below.

\vspace{1ex}

First, consider the family of (all, for the sake of definability) suitable affine coordinates $x_{1}, \ldots, x_{n}$, $n=k+l$, at $a \in M_{0}$,
i.e.\ such affine coordinates that $\partial^{d}g/\partial \, x_{n}^{d} \, (a) \neq 0$ with $d := \mu_{a}(g)$. More precisely, two kinds of variables:
$\xi_{1},\ldots,\xi_{k}$ and $\rho_{1},\ldots,\rho_{l}$ occur here; the first ones vary over the closed unit ball $K^{\circ}$ and the second ones over the open unit ball $K^{\circ \circ}$. We can thus consider, among others, the family of affine coordinates of the form
$$ \xi_{1}' = \xi_{1} + u_{1} \rho_{l}, \ \ldots \ , \xi_{k}' = \xi_{1} + u_{k} \rho_{l}, \ \rho_{1}' = \rho_{1} + v_{1} \rho_{l}, \ \ldots \ ,
   \rho_{l}' = \rho_{l} + v_{1} \rho_{l}, $$
with $u_{1},\ldots,u_{k},v_{1},\ldots,v_{l} \in K^{\circ}$. For simplicity, we shall further write the coordinates $x_{1},\ldots,x_{n}$, considering the definable family of all suitable coordinates (coming from the affine ones in the ambient space), which of course depend on the point $a$. Finally, set
$$ N_{1}(a) = V( \partial^{d-1}g/\partial \, x_{n}^{d-1}), \ \ \mathcal{E}_{1}(a) = \emptyset , $$
$$ \mathcal{H}_{1}(a) = \left\{ (\partial^{q}g/\partial \, x_{n}^{q}|N_{1}(a), \, d-q), \; q=0,1,\ldots,d-2 \right\}, $$
and
$$ \nu_{2}(a) = \mu_{\mathcal{H}_{1}(a)} := \min \left\{ \frac{\mu_{a}(\partial^{q}g/\partial \, x_{n}^{q}|N_{1}(a))}{d-q}, \  q=0,\ldots,d-2 \right\}. $$
Notice that $N_{1}(a)$ can be regarded both as a codimension 1 submanifold in the open subset $M_{0} \setminus V(\partial^{d}g/\partial \, x_{n}^{d})$ (which is
beneficial for the analysis of definability) or as its germ (which is the case treated originally in the theory of infinitesimal presentations, op.cit.).
It follows directly from Proposition~\ref{order} that the entry $\nu_{2}$ takes only finitely many values and hence its equimultiple strata are definable.
Again, further definable constructions should take into account those strata. The same holds over each field $L$ elementarily equivalent to $K$ in the language $\mathcal{L}_{\mathcal{A}(K)}$ (with the same set of orders of vanishing).

\vspace{1ex}

Construction~4.23, op.cit., yields the codimension 1 presentation:
$$ \mathcal{F}_{2}(a) = \mathcal{G}_{2}(a) := $$
$$ \left\{ (\partial^{q}g/\partial \, x_{n}^{q}|N_{1}(a), \, (d-q) \cdot \nu_{2}(a)), \, q=0,1,\ldots,d-2 \right\}, $$
which  satisfies the conditions of Proposition~4.12, op.cit.; in particular, $\mu_{\mathcal{F}_{2}(a)} =1$. Why the construction falls into the three stages $\mathcal{G}$, $\mathcal{F}$ and $\mathcal{H}$ will be clear in the next years of the process.

\vspace{1ex}

Next, repeat Construction~4.18, op.cit. To this end,
consider again the family of suitable coordinates on $N_{1}(a)$ induced by generic affine coordinates on the ambient space, taking also into account the strata on which given pairs $(h,\mu_{h}) \in \mathcal{F}_{2}(a)$ satisfy the condition $\mu_{a}(h) = \mu_{h}$. In this way, we get a codimension 2 presentation $\mathcal{H}_{2}(a)$ determined by some definable data expressed in terms partial derivatives with respect to the definable family of suitable coordinates.

\vspace{1ex}

The resolution process will be continued until $\nu_{t+1} = 0$ or $\infty$, which must happen for a $t \leq n$. In year zero, however, we eventually get the invariant $\mathrm{inv}_{X}(a)$ of the form $(\ldots;\infty)$, whose maximum stratum $S = C_{0}$ is an analytic submanifold of $M_{0}$. After blowing up the stratum $S = C_{0}$, we pass to the next year.

\begin{remark}
The analysis on the successive spaces $M_{j}$, $j \geq 1$, comes down to the case of affine ambient spaces with affine coordinates via the standard charts on the projective spaces involved when blowing up.
\end{remark}

Suppose now that the process has been carried out in the years $0,1,2,\ldots,j$.

\vspace{2ex}

{\em Year $(j+1)$.} We have thus constructed the following sequence of blowups (op.cit., Section~1):
$$ \sigma_{j} : M_{j} \to M_{j-1}, \ \sigma_{j-1} : M_{j-1} \to M_{j-2}, \ \dots \ , \ \sigma_{1} : M_{1} \to M_{0}; $$
the centers $C_{k-1}$ of $\sigma_{k}$ are admissible and the exceptional hypersurfaces $E_{k}$ on $M_{k}$ and $C_{k}$ simultaneously have only normal crossings.

\vspace{1ex}

As before, for each $a \in M_{j}$, we start with the following codimension 0 presentation for the transform $g_{1}$ of $g$ under $\sigma_{1} \circ \ldots \circ \sigma_{j}$:
$$ (N_{0}(a), \mathcal{G}_{1}(a), \mathcal{E}_{1}(a)), \ \ N_{0}(a) = M_{1}, \ \ \mathcal{G}_{1}(a) = \{ (g_{1}, \mu_{a}(g_{1})) \},  $$
where $N_{0} = M_{1}$ and $\mathcal{E}_{1}(a))$ is defined as follows:

let $\nu_{1} := \mu_{a}(g_{1}))$, $i = i(a) \leq j$ be the smallest $k$ with $\nu_{1}(a) = \nu_{1}(a_{k})$,
$$ E^{1}(a) := \{ H \in E(a): \ H \ \text{is the transform of some element of} \ E(a_{i}) \}, $$

$s_{1}(a) := \sharp \ E^{1}(a)$ and $\mathcal{E}_{1}(a)) := E(a) \setminus E^{1}(a)$.

\noindent
Since the invariant $\mathrm{inv}_{X}(\cdot)$ constructed in the previous years takes only finitely many values and is both upper-semicontinuous and infinitesimally upper-semicontinuous, it is not difficult to check that the equimultiple strata of the invariant $i(\cdot)$ are definable, whence so are the families $E^{1}(\cdot)$ and $\mathcal{E}_{1}(\cdot)$.

\vspace{1ex}

Next, let $\mathcal{F}_{1}(a)$ be $\mathcal{G}_{1}(a)$ together with all pairs $(f,\mu_{f}) = (\theta_{H},1)$ with $H \in E^{1}(a)$, where $\theta_{H}$ is an analytic equation of $H$. By Corollary~\ref{zero-loc}, $\mathcal{F}_{1}(a)$ is determined by definable data. Now, apply Construction~4.18, op.cit., as in the year zero, to get a codimension 1 presentation
$$ (N_{1}(a), \mathcal{H}_{1}(a), \mathcal{E}_{1}(a)), $$
which is determined by definable data as well. Then
$$ \mu_{2}(a) := \mu_{\mathcal{H}_{1}(a)} = \infty \ \ \text{iff} \ \ \mathcal{H}_{1}(a) = 0. $$
If $\mu_{2}(a) < \infty$, set
$$ \mu_{2,H} := \min \left\{ \frac{\mu_{H,a}(h)}{\mu_{h}}: \ (h,\mu_{h}) \in \mathcal{H}_{1}(a) \right\}, \ \ H \in \mathcal{E}_{}(a) $$
and
$$ \mu_{2}(a) := \mu_{\mathcal{H}_{1}(a)}, \ \ \ \nu_{2}(a) := \mu_{2}(a) - \sum_{H} \mu_{2,H}(a). $$
By Proposition~\ref{ord-part}, the invariant $\nu_{2}(\cdot)$ takes only finitely many values and its equimultiple loci are definable.

If $\nu_{2}(a) = 0$ or $\infty$, set $\mathrm{inv}_{X}(a) := (\nu_{1}(a),s_{1}(a); \nu_{2}(a))$. Otherwise, apply Construction~4.23, op.cit., to get a codimension 1 presentation
$$ (N_{1}(a), \mathcal{G}_{2}(a), \mathcal{E}_{1}(a)) \ \ \ \text{with} \ \ \ \mu_{\mathcal{G}_{2}(a)} =1. $$
The construction consists in dividing the $h \in \mathcal{H}_{1}(a)$, previously scaled so that the $\mu_{h}$ are equal, by their greatest common divisor that is a monomial in the equations $\theta_{H}$ of $H \in \mathcal{E}_{1}(a)$. Hence and again by  Proposition~\ref{ord-part}, $\mathcal{G}_{2}(a)$ is determined by definable data.

\vspace{1ex}

Now, let $i = i(a) \leq j$ be the smallest $k$ such that
$$ (\nu_{1}(a),s_{1}(a);\nu_{2}(a)) = (\nu_{1}(a_{k}),s_{1}(a_{k});\nu_{2}(a_{k})), $$
$$ E^{2}(a) := \{ H \in \mathcal{E}_{1}(a): \ H \ \text{is the transform of some element of} \ \mathcal{E}_{1}(a_{i}) \}, $$
$s_{2}(a) := \sharp \ E^{2}(a)$ and $\mathcal{E}_{2}(a)) := \mathcal{E}_{1}(a) \setminus E^{1}(a)$. Then
$$ (N_{1}(a), \mathcal{G}_{2}(a), \mathcal{E}_{2}(a)) $$
is a codimension 1 presentation determined by definable data as well.

\vspace{1ex}

Next, let $\mathcal{F}_{2}(a)$ be $\mathcal{G}_{2}(a)$ together with all pairs $(f,\mu_{f}) = (\theta_{H},1)$ with $H \in E^{2}(a)$, where $\theta_{H}$ is an analytic equation of $H$. The process continues inductively until $\nu_{t+1} = 0$ or $\infty$ for a $t \leq n$, and eventually yields the invariant $\mathrm{inv}_{X}(\cdot)$ on $M_{j}$ which takes only finitely many values and whose equimultiple loci
$$ S_{X}(a) := \{ x \in M_{j}: \ \mathrm{inv}_{X}(x) = \mathrm{inv}_{X}(a) \}, \ \ a \in M_{j}, $$
are definable; $S_{X}(a)$ will also be regarded as a germ at $a$. Its maximum stratum $S$ is an analytic submanifold or a normal crossing submanifold according as its maximum value is $((\ldots;\infty)$ or $(\ldots;0)$. In the latter case, for any $a \in S$, the irreducible components $Z$ of $S_{X}(a)$ are of the form (op.cit., Theorem~1.14):
\begin{equation}\label{comp}
   Z = S_{X}(a) \cap \, \bigcap \, \{ H \in E(a): \ Z \subset H \}.
\end{equation}
Then, in order to eventually achieve a smooth maximum stratum, the invariant should be extended as outlined below.

\vspace{1ex}

Consider any total ordering on the collection of all subsets $I$ of $E_{j}$. Observe that whether the intersection $S_{X}(a) \cap \bigcap I$ is an analytic submanifold at $a$ is a definable property with respect to the points $a$ (which is expressed, in view of equality~\ref{comp}, by means of suitable coordinate projections). Therefore the components $Z$ of $S_{X}(a)$ can be defined by the following formula~(*):

\vspace{1ex}

\begin{em}
$Z = S_{X}(a) \cap \bigcap I$ is an analytic  submanifold at $a$ for some $I$ and, for every $J$, if $Z \subset S_{X}(a) \cap \bigcap J$ is an analytic submanifold, then $Z = S_{X}(a) \cap \bigcap J$.
\end{em}

\vspace{1ex}

The family of the components $Z$ of $S_{X}$ at the points $a$ is thus definable (consider the product of $\sharp \, E_{j}$ copies of $M_{j}$). For a component $Z$ at $a$, let $J(Z)$ be the set of all $H \in E_{j}$ containing $Z$. Set
$$ J(a) := \max \{ J(Z): \ Z \ \ \text{a component of} \ \ S_{X}(a) \}. $$
Then the index $J(a)$ is definable:

\vspace{1ex}

\begin{em}
$J(a) = I^{*}$ iff formula~(*) holds for $I^{*}$ at $a$ and, for every $I \varsupsetneq I^{*}$ and $J > I^{*}$, formula~(*) holds at $a$ neither for $I$ nor for $J$.
\end{em}

\vspace{1ex}

Extend the invariant on $M_{j}$  by putting
$$ \mathrm{inv}_{X}^{e}(a) := (\mathrm{inv}_{X}(a);J(a)). $$
Then the maximum locus of $\mathrm{inv}_{X}^{e}(\cdot)$ is smooth (op.cit., Remark~1.15). Actually, for any component $Z$ of the maximum locus of $\mathrm{inv}_{X}(\cdot)$ at a point $a$, one can choose an ordering above so that
$$ J(Z) = J(a) = \max E_{j}. $$
Therefore the component $Z$ extends to an analytic submanifold of $M_{j}$, the maximum locus of $\mathrm{inv}_{X}^{e}(\cdot)$. Furthermore, by choosing a suitable ordering on the subsets of each $E_{j}$, one achieves the extended invariant $\mathrm{inv}_{X}^{e}(\cdot)$ with the property that every germ $S_{\mathrm{inv}_{X}^{e}(a)}$ is smooth (op.cit., Remark~1.16). Hence the maximum locus of $\mathrm{inv}_{X}^{e}(\cdot)$ is an analytic submanifold of the ambient space (which means strong analytic, by the convention adopted in Section~2).

\vspace{2ex}

{\em Sketch of resolution of singularities.} Now we briefly outline the desingularization algorithm in the hypersurface case (op.cit., Theorem~1.6), which immediately yields transformation to normal crossings (op.cit., Theorem~1.10) as well. The proof, given in op.cit., Section~10, carries over verbatim to the definable settings treated here. It essentially relies on that the (extended) desingularization invariant takes only finitely many values and, though those values $\nu_{r}(\cdot)$ are merely rational numbers, it behaves as if those values were integers (unless $\infty$). This directly follows from the infinitesimal upper-semicontinuity of the invariant and its finitary character in each particular year of the process along with the estimates of denominators given below (op.cit., p.~214).

\vspace{1ex}

In each year of the process, the entries $\nu_{r}(a)$, $r = 2, \ldots, t \leq n$, are quotients of positive integers whose denominators are bounded in terms of the previous part of the invariant $\mathrm{inv}_{X}(a)$. More precisely, define recursively $e_{2}(a) := \nu_{1}(a)$ and $e_{r+1}(a) := \max \{ e_{r}(a)!, e_{r}(a)! \cdot \nu_{r}(a) \}$. Then \ $e_{r}(a)! \cdot \nu_{r}(a) \in \mathbb{N}$ \ and \ $e_{t+1}(a)! \cdot \mu_{t+1}(a) \in  \mathbb{N}$.

\vspace{1ex}

In each year $j$ of the process, the maximum locus $C_{j}$ of the invariant $\mathrm{inv}_{X}$ (or the extended invariant $\mathrm{inv}_{X}^{e}$ if $\nu_{t+1} = 0$ on the maximum locus of $\mathrm{inv}_{X}$) is smooth so that one can blow it up. For each point $a \in C_{j}$, if $\mathrm{inv}_{X}(a) = (\ldots ; \infty)$, then
$$ \mathrm{inv}_{X}(a') < \mathrm{inv}_{X}(a) \ \ \text{for all} \ \ a' \in \sigma_{j+1}^{-1}(a). $$
Otherwise (op.cit., Theorem~1.14), we get
$$ (\mathrm{inv}_{X}(a'),\mu_{X}(a')) < (\mathrm{inv}_{X}(a),\mu_{X}(a)) \ \ \text{for all} \ \ a' \in \sigma_{j+1}^{-1}(a), $$
where $\mu_{X}(a) = \mu_{t+1}(a)$ if $\nu_{t+1}(a) = 0$. Hence the maximum value of the invariant must decrease after a finite number of admissible blowups and, eventually, the transform $X_{j}$ becomes smooth.

\vspace{1ex}

However, some further admissible blowups may be needed in order to satisfy the requirement that the final transform $X_{k}$ and $E_{k}$ simultaneously have only normal crossings. To this end, one must blow up the successive maximum strata of the invariant $\mathrm{inv}_{X}$ until its parameter $s_{1}$ has decreased to zero everywhere on $X_{k}$ (op.cit., p.~285). Then we attain the final step of the desingularization process.  \hspace*{\fill} $\Box$

\vspace{1ex}

In a similar manner, we are able to achieve a definable version of transforming to normal crossings a sheaf of ideals $\mathcal{I} = \mathcal{I}_{0} \subset \mathcal{O}_{M_{0}}$ generated by a finite number of strong analytic functions $f_{1},\ldots,f_{s}$ on $M_{0}$. This process uses the successive weak transforms $\mathcal{I}_{j}$ of the ideal $\mathcal{I}$ when blowing up the maximal strata of the desigularization invariant (op.cit., Theorem~1.10). We adopt the previous notation and, for convenience, remind the reader the statement.

\begin{theorem}\label{main-2}
Under the above assumptions,  there exists a finite sequence of blowups with smooth admissible centers $C_{j}$ such that the final weak transform of $\mathcal{I}$ is $\mathcal{I}_{k} = \mathcal{O}_{M_{k}}$ and the pull-back $\sigma^{-1}(\mathcal{I}) \cdot \mathcal{O}_{M_{k}} = E_{k}$ of the sheaf of ideals $\mathcal{I}$ is a normal crossing divisor; here $\sigma$ is the composite of the $\sigma_{j}$.   \hspace*{\fill} $\Box$
\end{theorem}

\vspace{1ex}

\section{Application to the problem of definable retractions}

In this section, we demonstrate applications of definable resolution of singularities to the problems of definable retractions and extending continuous definable functions. The main aim here is the following theorem on the existence of definable retractions onto an arbitrary closed definable subset, whereby definable non-Archimedean versions of the extension theorems by Tietze--Urysohn and Dugundji follow directly (cf.~\cite{Now-4,Now-5}, where also conducted is a more detailed discussion about their classical, purely topological counterparts).

\begin{theorem}\label{ret}
Consider definable, strong analytic functions $g_{1},\ldots,g_{r}$ on a strong analytic manifold $M$. Let $X := V(g_{1},\ldots,g_{r})$ be their zero locus and $A$ be a closed definable subset of $X$. Then there exists a definable retraction $X \to A$.
\end{theorem}

We immediately obtain

\begin{corollary}\label{ret-2}
For each closed definable subset $A$ of $K^{n}$,
there exists an definable retraction $K^{n} \to A$.   \hspace*{\fill} $\Box$
\end{corollary}

The case of analytic structures, determined on complete rank one valued fields $K$ by separated power series, was already established in our previous paper~\cite[Theorem~1]{Now-5}. Using the results of this paper, we can carry out that proof to the general settings of separated analytic structure as outlined below. Our proof made use of the following basic tools:

$\bullet$ elimination of valued field quantifiers (due to Cluckers--Lipshitz--Robinson~\cite{L-R,C-Lip-R,C-Lip-0,C-Lip});

$\bullet$ embedded resolution of singularities and transforming an ideal to normal crossings by blowing up (due to Bierstone--Milman~\cite{BM} or Temkin~\cite{Tem-2});

$\bullet$ the technique of quasi-rational and $R$-subdomains (due to Lipshitz--Robinson~\cite{L-R-0});

$\bullet$ and the closedness theorem~\cite{Now-1,Now-2,Now-3}.

\begin{remark}
Observe that the advantage of working here with the more flexible, strong analytic settings lies also in that we do not need to appeal to the theory of quasi-affinoid subdomains.
\end{remark}

Now, we are able, after some elaboration, to repeat that previous proof, via the definable version of transformation to normal crossings treated here, except for~\cite[Lemma~3.1]{Now-5} recalled below, because the full version of resolution of singularities seems to be unavailable in the definable settings.

\begin{lemma}\label{ind-var}
Let $Z \varsubsetneq X$ be two closed subvarieties of $M$ and $A$ a closed definable subset of $Z$. Suppose that $X$ is non-singular of dimension $N$ and Theorem~\ref{ret} holds for closed definable subsets of every non-singular variety of this kind of dimension $< N$. Then there exists an definable retraction $r:Z \to A$.
\end{lemma}

In our paper~\cite{Now-5}, this lemma holds in full generality. But in the proof of Theorem~\ref{ret}, it was involved in an induction procedure and used only when $Z$ was the zero locus of one analytic function
$$ \psi_{j} = (\psi_{j-1} \circ \sigma_{j}) \cdot \chi_{j} \ \ \text{with} \ \ j=0,\ldots,k, $$
thus being an analytic hypersurface of $M$ in the algebro-geometric sens. (By abuse of notation, we often use the same letter for an analytic subvariety and its support, i.e.\ underlying topological space, which does not lead to confusion.) Hence it suffices to prove here the following version (where $s=1$ would be enough):

\begin{lemma}\label{ind-var-1}
Let $X$ be a closed, strong analytic submanifold of $M$ of dimension $N$, $f_{1},\ldots,f_{s}$ be strong analytic functions on $X$, $Z := V(f_{1},\ldots,f_{s})$ and
$A$ be a closed definable subset of $Z$. Suppose $Z$ is of dimension $< N$ and that Theorem~\ref{ret} holds for closed definable subsets of every closed, strong analytic submanifold of $M$ of dimension $< N$. Then there exists an definable retraction $r:Z \to A$.
\end{lemma}

\begin{proof}
Apply Theorem~\ref{main-2} to transform to normal crossings the sheaf of ideals $\mathcal{I}$ generated by $f_{1},\ldots,f_{s}$ on the analytic manifold $X$. Set
$$ \tau_{j} := \sigma_{1} \circ \ldots \sigma_{j}, \ \  j=1,\ldots,k, \ \ \text{and} \ \ A^{\tau} := \tau^{-1}(A). $$
Then $Z^{\tau_{k}} = \bigcup E_{k}$. Considering the canonical map from the disjoint union $\coprod E_{k}$ onto $ \bigcup E_{k}$ and using the assumption of the lemma, it is not difficult to check that there is a definable retraction $\rho_{k}: Z^{\tau_{k}} \to A^{\tau_{k}}$.

Therefore, by op.cit., Corollary~2.13, there is a definable retraction $r_{k-1}: Z^{\tau_{k-1}} \to (C_{k-1} \cup A^{\tau_{k-1}})$. Again by the assumption, there is a definable retraction $C_{k-1} \to (C_{k-1} \cap A^{\tau_{k-1}})$, and hence a definable retraction $\rho_{k-1}: Z^{\tau_{k-1}} \to A^{\tau_{k-1}}$.

As before, by op.cit., Corollary~2.13, there is a definable retraction $r_{k-2}: Z^{\tau_{k-2}} \to C_{k-2} \cup A^{\tau_{k-2}})$.  Again by the assumption, there is a definable retraction $C_{k-2} \to (C_{k-2} \cap A^{\tau_{k-2}})$, and hence a definable retraction $\rho_{k-2}: Z^{\tau_{k-2}} \to A^{\tau_{k-2}}$.

Proceeding recursively, we eventually achieve a definable retraction $\rho_{0}: Z \to (A)$, we are looking for.
\end{proof}

\begin{remark}
It is plausible that the above results will also hold in more general settings of certain tame non-Archimedean geometries considered in the papers~\cite{Hal} and~\cite{C-Com-L}.
\end{remark}

Perhaps the strongest, purely topological, non-Archimedean results on retractions are those from the papers~\cite{Da} and~\cite{K-K-S} recalled below respectively.

\begin{theorem}
1) Any closed subset $A$ of an ultranormal metrizable space $X$ is a retract of $X$.

2) Any compact metrizable subset $A$ of an ultraregular space $X$ is a retract of $X$.
\end{theorem}

\section{Intricacies of non-Archimedean analytic geometry}

In this final section, we discuss some background behind quantifier elimination in non-Archimedean analytic geometry. The theory of semi- and sub-analytic sets was first developed over the real field (cf.~\cite{Loj,Gab,Hi}) with the three powerful tools: Gabrielov's~\cite{Gab} complement theorem (in other words, quantifier simplification for the real analytic structure), and Hironaka's~\cite{Hi} resolution of singularities and flattening of analytic morphisms by blowing up.

\vspace{1ex}

In the locally compact case, real and $p$-adic, even full quantifier elimination in a 1-sorted analytic language was established by Denef--van den Dries~\cite{De-Dries}. Similar techniques, when applied over algebraically closed, complete, rank one valued fields $K$, require the use of various G-topologies (cf.~\cite{B-G-R,F-Put}) because the underlying metric topology is totally disconnected and non-locally compact. An analogous quantifier elimination over those fields would be available if a global rigid analogue of Hironaka's flattening were valid. However, the proof of such an analogue given by Gardener--Schoutens~\cite{G-Schou} failed to be true, as it was indicated in the following counterexample by Lipshitz--Robinson~\cite{L-R-1}.

\begin{example}
Denote by
$$ T_{n} := \left\{ \sum a_{\nu} \xi^{\nu} : \ |a_{\nu}| \rightarrow 0 \ \text{as} \ |\nu| \rightarrow \infty \right\} $$
the ring of strictly convergent power series over $K$ in the variables $\xi = (\xi_{1},\ldots,\xi_{n}$. For $\alpha = (\alpha_{1},\ldots,\alpha_{n})$ with $\alpha_{i} \in |K \setminus \{ 0 \}|$, the ring
$$ T_{n,\alpha} := \left\{ \sum a_{\nu} \, \xi^{\nu} : \ |a_{\nu}\alpha^{\nu}| \rightarrow 0 \ \text{as} \ |\nu| \rightarrow \infty \right\} $$
is the affinoid algebra of the rational polydisc of polyradius $\alpha$.

Let $D$ be the disc of $K$-rational radius $\varepsilon$; then $\mathcal{O}(D) = T_{1,\varepsilon}$. Suppose
$$ f  \in \mathcal{O}(D) \setminus \bigcup_{\delta > \varepsilon} \, T_{1,\delta} $$
which means that $f$ converges on $D$ and is not overconvergent. For instance, take
$$ f := \sum_{n \geq 1} \, a^{n - n^{2}} \xi^{n^{2}} \ \ \text{with} \ \ |a| = \varepsilon <1. $$

Put
$$ X := \mathrm{Sp}\, T_{2} \ \ \text{and} \ \ Y := \mathrm{Sp}\, T_{4}/I, \ \ I := (\xi_{2} - f(a\xi_{4}), \xi_{1} - a\xi_{4}) \cap (\xi_{3}). $$

The Ref \cite[Theorem~5.4]{L-R-1} says that then the map
$$ \varphi: Y \to X, \ \ (\xi_{1},\xi_{2},\xi_{3},\xi_{4}) \mapsto (\xi_{1},\xi_{2}). $$
cannot be flattened by a finite sequence of local blowups. Its subtle proof relies on the concept of a flatificator at an analytic point defined in a wide affinoid neighbourhood, thus involving the theory of Berkovich spaces. This allows the authors to reduce the initial problem to that of the existence of a curve defined in an affinoid sub-polydisc of $X$ without analytic continuation to a larger polydisc.
\end{example}

Not only does the proof by Gardener--Schoutens have a serious gap, but also their quantifier elimination fails indeed. Given an algebraically closed, complete, rank one valued field $K$, Cluckers--Lipshitz~\cite[Theorem~4.3]{C-Lip} constructed a strictly convergent subanalytic subset $X \subset K^{3}$ which is not quantifier-free definable in the 1-sorted analytic language of strictly convergent analytic structures. Nevertheless, then quantifier elimination holds in the 1-sorted analytic language of separated analytic structures (cf.~\cite[Theorem~4.5.15]{C-Lip-0}).

\vspace{1ex}

Recently Ducros~\cite{Duc} develops (inspired by Raynaud--Gruson~\cite{Ray-Gru}) flattening techniques for Berkovich spaces over complete, rank one valued fields $K$. One of the essential ingredients of his approach (namely Lemma~1.18) is a consequence of Temkin's version of the Gerritzen--Grauert theorem~\cite[Theorem~3.1]{Tem-1}. Using those techniques, he proves Theorem~7.8 that the image of a morphism between compact analytic spaces is a finite union of the images of maps each of which is a finite composite of blowups and quasi-\'{e}tale morphisms. Eventually, Ducros anticipates that it geometrically corresponds, if the ground field $K$ is algebraically closed, to quantifier elimination for the separated analytic structure on $K$ (which is a definitional extension of the strictly convergent one by solutions of certain polynomial Henselian systems considered by Cluckers--Lipshitz~\cite{C-Lip}).

\vspace{1ex}

We conclude the paper with some comments. The collections of $A_{m,n}$ and $A_{m,n}^{\dag}$ correspond respectively to the collections of $S_{m,n}^{\circ}$ and $S_{m,n}$, which were earlier studied in the paper~\cite{L-R-0} in the case of complete, rank one valued fields $K$. Since the rings $A_{m,n}^{\dag} = S_{m,n}$ have good algebraic properties, we were able in our previous paper~\cite{Now-5} to use the classical version of canonical desingularization (along with the theory of quasi-affinoid subdomains).

\vspace{1ex}

Generally, however, separated analytic structures admit reasonable quantifier elimination, but usual resolution of singularities from rigid analytic geometry is not available. The opposite situation holds for strictly convergent analytic structures. It is thus of great importance that definable desingularization for the former structures, provided in this paper, makes both these powerful tools of analytic geometry available at the same time. And let us emphasize once again that the work within strong analytic manifolds and maps allows us not to appeal to the theory of quasi-affinoid subdomains.

\vspace{1ex}

A further direction of research might be into definable Lipschitz retractions and extending definable Lipschitz continuous functions (perhaps with the same Lipschitz constant) over non-locally compact fields. These as yet open problems may be investigated in analytic structures and in the tame non-Archimedean geometries from the papers~\cite{Hal} and~\cite{C-Com-L} as well. Extending Lipschitz continuous functions $f: A \to \mathbb{R}$, with the same Lipschitz constant from a subset $A$ of $\mathbb{R}^{n}$, goes back to McShane and Whitney. The more difficult case of functions with values in $\mathbb{R}^{k}$ was achieved by Kirszbraun. Aschenbrenner--Fischer~\cite{Asch} obtained a definable version of Kirszbraun's theorem for definably complete expansions of ordered fields. Recently Cluckers--Martin~\cite{C-Ma} established a $p$-adic version of Kirszbraun's theorem. They proved it, along with the existence of a definable Lipschitz retraction (with constant $1$) for any closed definable subset $A$ of $\mathbb{Q}_{p}^{n}$, proceeding with simultaneous induction on the dimension $n$ of the ambient space. To this end, they introduced a certain form of preparation cell decompositions with Lipschitz continuous centers. Besides, their construction of definable retractions makes use of some definable Skolem functions. Therefore we cannot expect that their approach can be directly carried over to geometry over non-locally compact Henselian fields, where cells are no longer finite in number (but parametrized by residue field variables) and definable Skolem functions do not exist in general. The non-locally compact case will certainly require a new approach and ingenious ideas.

\vspace{1ex}

Let me finally mention that my work in non-Archimedean geometry was inspired by the joint paper with J.~Koll\'{a}r~\cite{K-N}, which deals with the very concept and extension of continuous hereditarily rational functions on real and $p$-adic varieties, and the results of which were further carried over to non-locally compact fields in my papers~\cite{Now-1,Now-2}.

\vspace{1ex}

\begin{small}
Institute of Mathematics, \ Faculty of Mathematics and Computer Science

Jagiellonian University

ul.~Profesora S.\ \L{}ojasiewicza 6, \ 30-348 Krak\'{o}w, Poland

{\em E-mail address: nowak@im.uj.edu.pl}

ORCID ID: https://orcid.org/0000-0001-8070-032X
\end{small}


\vspace{1ex}

\begin{thebibliography}{99}

\bibitem{Asch}
Aschenbrenner M.; Fischer A. Definable versions of theorems
by Kirszbraun and Helly. {\em Proc.\ Lond.\ Math.\ Soc.} (3) {\bf 2011}, 102, 468--502.

\bibitem{Bas}
Basarab, S.A. Relative elimination of quantifiers for Henselian valued fields.
{\em Ann.\ Pure Appl.\ Logic} {\bf 1991}, 53, 51--74.

\bibitem{BM}
Bierstone, E.; Milman, P.D. Canonical desingularization in
characteristic zero by blowing up the maximum strata of a local
invariant. {\em Inventiones Math.} {\bf 1997}, 128, 207--302.

\bibitem{B-G-R}
Bosch, S.; G\"{u}ntzer, U.; Remmert, R. {\em Non-Archimedian
Analysis: a systematic approach to rigid analytic geometry};
Grundlehren der math.\ Wiss.; Springer: Berlin/Heidelberg, Germany, 1984; Volume 261.

\bibitem{C-Com-L}
Cluckers, R.; Comte, G.; Loeser, F. Non-Archimedean
Yomdin--Gromov para\-metrizations and points of bounded height.
{\em Forum Math.} {\bf 2015}, 3 , 60 pp.

\bibitem{C-Lip-R}
Cluckers, R.; Lipshitz, L.; Robinson, Z. Analytic cell
decomposition and analytic motivic integration. {\em Ann.\ Sci.\ \'{E}cole
Norm.\ Sup.\ (4)} {\bf 2006}, 39, 535--568.

\bibitem{C-Lip-0}
Cluckers, R.; Lipshitz, L. Fields with analytic structure.
{\em J.\ Eur.\ Math.\ Soc.} {\bf 2011}, 13, 1147--1223.

\bibitem{C-Lip}
Cluckers, R.; Lipshitz, L. Strictly convergent analytic
structures. {\em J.\ Eur.\ Math.\ Soc.} {\bf 2017}, 19, 107--149.

\bibitem{C-Ma}
Cluckers, R.; Martin F. A definable p-adic analogue of
Kirszbraun's theorem on extension of Lipschitz maps. {\em J.\ Inst.\
Math.\ Jussieu} {\bf 2018}, 17, 39--57.

\bibitem{Da}
Dancis, J. Each closed subset of metric space $X$ with $\mathrm{Ind}\, X = 0$ is a retract.
{\em Houston J. Math.} {\bf 1993}, 19, 541--550.

\bibitem{De-Dries}
Denef, J.; van den Dries, L. $p$-adic and real subanalytic
sets. {\em Ann.\ Math.} {\bf 1988}, 128, 79--138.

\bibitem{Dries}
van den Dries, L. Analytic Ax--Kochen--Ershov theorems.
{\em Contemporary Mathematics} {\bf 1992}, 131, 379--392.

\bibitem{Dries-Has}
van den Dries, L.; Haskell, D.; Macpherson, D. One dimensional
$p$-adic subanalytic sets. {\em J.\ London Math.\ Soc.} {\bf 1999}, 56, 1--20.

\bibitem{Dries-Mac}
van den Dries, L.; Macintyre, A.; Marker, D. The elementary
theory of restricted analytic fields with exponentiation. {\em Ann.\
Math.} {\bf 1994}, 140, 183--205.

\bibitem{Duc}
Ducros, A. D\'{e}visser, d\'{e}couper, \'{e}clater and aplatir les espaces de Berkovich.
arXiv:1906.00301 [math.AG], {\bf 2019}.

\bibitem{El-1}
Ellis, R.L. A non-Archimedean analogue of the Tietze-Urysohn
extension theorem. {\em Indag.\ Math.} {\bf 1967}, 29, 332--333.

\bibitem{El-2}
Ellis, R.L. Extending continuous functions on
zero-dimensional spaces. {\em Math.\ Ann.} {\bf 1970}, 186, 114--122.

\bibitem{F-Put}
Fresnel, J; van der Put, M. {\em Rigid Analytic Geometry and its Applications}; Progress in Math.; Birkh\"{a}user; Boston, USA, 2004; Volume 218.

\bibitem{Gab}
Gabrielov, A. Projctions of semianalytic sets. {\em Funct.\ Anal.\ Appl.} {\bf 1968}, 2, 282--291.

\bibitem{G-Schou}
Gardener, T.; Schoutens, H. Flattening and subanalytic sets in rigid analytic geometry.
{\em Proc.\ London Math.\ Soc.} (3) {\bf 2001}, 83, 681--707.

\bibitem{Hal}
Halupczok, I. Non-Archimedean Whitney stratifications.
{\em Proc.\ London Math.\ Soc.} (3) {\bf 2014}, 109, 1304--1362.

\bibitem{Hi}
Hironaka, H. {\em Introduction to Real Analytic Sets and Real Analytic Maps.};
Istituto Matematico "L.~Tonelli" dell' Universit\`{a} di Pisa; Pisa, Italy, 1973.

\bibitem{K-K-S}
K\c{a}kol, J; Kubzdela, A.; \'{S}liwa, W. A non-Archimedean
Dugundji extension theorem. {\em Czechoslovak Math.\ J.} {\bf 2013}, 63 (138), 157--164.

\bibitem{K-N}
Koll\'{a}r, J.; Nowak, K. Continuous rational functions on
real and $p$-adic varieties. {\em Math.\ Zeitschrift} {\bf 2015}, 279, 85--97.

\bibitem{Ku-1}
Kuhlmann, F.-V. Quantifier elimination for Henselian fields relative to additive and multiplicative congruences.
{\em Israel J.\ Math.} {\bf 1994}, 85, 277--306.

\bibitem{Lip}
Lipshitz, L. Rigid subanalytic sets. {\em Amer.\ J.\ Math.} {\bf 1993}, 115, 77-108.

\bibitem{L-R-0}
Lipshitz, L.; Robinson, Z. Rings of Separated Power Series
and Quasi-Affinoid Geometry. {\em Ast\'{e}risque} {\bf 2000}, 264.

\bibitem{L-R}
Lipshitz, L.; Robinson, Z. Uniform properties of rigid
subanalytic sets. {\em Trans.\ Amer.\ Math.\ Soc.} {\bf 2005}, 357, 4349--4377.

\bibitem{L-R-1}
Lipshitz, L.; Robinson, Z. Flattening and analytic continuation of affinoid morphisms:
remarks on a paper of Gardener and Schoutens. {\em Proc.\ London Math.\ Soc.} (3) {\bf 2005}, 91, 443--458.

\bibitem{Loj}
\L{}ojasiewicz, S. {\em Ensembles Semi-Analytiques}; Institut I.H.E.S., Bures-sur-Yvette, France, 1965.

\bibitem{Now-1}
Nowak, K.J. Some results of algebraic geometry over Henselian
rank one valued fields. {\em Sel.\ Math.\ New Ser.} {\bf 2017}, 23, 455--495.

\bibitem{Now-2}
Nowak, K.J. A closedness theorem and applications in geometry
of rational points over Henselian valued fields. arXiv:1706.01774
[math.AG] {\bf 2017}.

\bibitem{Now-3}
Nowak, K.J. Some results of geometry over Henselian fields
with analytic structure. arXiv:1808.02481 [math.AG] {\bf 2018}.

\bibitem{Now-4}
Nowak, K.J. Definable retractions and a non-Archimedean
Tietze--Urysohn theorem over Henselian valued fields.
arXiv:1808.09782 [math.AG] {\bf 2018}.

\bibitem{Now-5}
Nowak, K.J. Definable retractions over complete fields with
separated power series. arXiv:1901.00162 [math.AG] {\bf 2019}.

\bibitem{Pa1}
Pas, J. Uniform p-adic cell decomposition and local zeta
functions. {\em J.\ Reine Angew.\ Math.} {\bf 1989}, 399, 137--172.

\bibitem{Ray-Gru}
Raynaud, M.; Gruson, L. Crit\`{e}res de platitude et projectivit\'{e}
techniques de "platification" d'un module. {\em Invent.\ Math.} {\bf 1971}, 13, 1-89.

\bibitem{Se}
Serre J.P. {\em Lie Algebras and Lie Groups}; Lecture Notes in
Mathematics; Springer: Berlin/Heidelberg, Germany, 2006; Volume 1500.

\bibitem{Tem-1}
Temkin, M. A new proof of the Gerritzen--Grauert theorem. {\em Math.\ Ann.} {\bf 2005}, 333, 261--269.

\bibitem{Tem-2}
Temkin M. Functorial desingularization over $\mathbb{Q}$:
boundaries and the embedded case. {\em Israel J.\ Math.} {\bf 2018}, 224, 455--504.

\end{thebibliography}
\end{document}